\documentclass[a4,12pt]{amsart}
\usepackage{amssymb,amstext,amsmath,amscd,amsthm,amsfonts,enumerate,latexsym}
\usepackage{color}
\usepackage[dvipdfmx]{graphicx}
\usepackage[all]{xy}

\usepackage{ytableau}

\theoremstyle{plain}
\newtheorem{thm}{Theorem}[section]

\newtheorem*{thm*}{Theorem}
\newtheorem*{cor*}{Corollary}

\newtheorem{prop}[thm]{Proposition}

\newtheorem{lem}[thm]{Lemma}
\newtheorem{cor}[thm]{Corollary}

\newtheorem*{claim*}{Claim}

\theoremstyle{definition}
\newtheorem{defn}[thm]{Definition}

\newtheorem{ex}[thm]{Example}

\newtheorem*{case*}{Case}

\theoremstyle{remark}

\numberwithin{equation}{thm}

\def\Ext{\operatorname{Ext}}

\def\Hom{\operatorname{Hom}}

\def\mod{\mathrm{mod}}

\def\m{\mathfrak m}

\newcommand{\rma}{\mathrm{a}}

\newcommand{\rmc}{\mathrm{c}}

\newcommand{\rme}{\mathrm{e}}

\newcommand{\rmo}{\mathrm{o}}

\newcommand{\rmr}{\mathrm{r}}

\newcommand{\rmC}{\mathrm{C}}

\newcommand{\rmQ}{\mathrm{Q}}

\newcommand{\calX}{\mathcal{X}}
\newcommand{\calY}{\mathcal{Y}}

\newcommand{\fkc}{\mathfrak{c}}

\newcommand{\mapright}[1]{%
\smash{\mathop{%
\hbox to 1cm{\rightarrowfill}}\limits^{#1}}}

\newcommand{\mapleft}[1]{%
\smash{\mathop{%
\hbox to 1cm{\leftarrowfill}}\limits_{#1}}}

\def\AGL{\operatorname{AGL}}

\def\gr{\mbox{\rm gr}}

\title{Ulrich ideals in the ring $k[[t^5,t^{11}]]$}

\author[Naoki Endo]{Naoki Endo}
\address{Department of Mathematics, Faculty of Science, Tokyo University of Science, 1-3 Kagurazaka, Shinjuku, Tokyo 162-8601, Japan}
\email{nendo@rs.tus.ac.jp}
\urladdr{https://www.rs.tus.ac.jp/nendo/}

\author[Shiro Goto]{Shiro Goto}
\address{Department of Mathematics, School of Science and Technology, Meiji University, 1-1-1 Higashi-mita, Tama-ku, Kawasaki 214-8571, Japan}
\email{shirogoto@gmail.com}

\author[Shin-ichiro Iai]{Shin-ichiro Iai}
\address{Mathematics laboratory, Sapporo College, Hokkaido University of Education, 1-3 Ainosato 5-3, Kita-ku, Sapporo 002-8502, Japan}
\email{iai@sap.hokkyodai.ac.jp}

\author[Naoyuki Matsuoka]{Naoyuki Matsuoka}
\address{Department of Mathematics, School of Science and Technology, Meiji University, 1-1-1 Higashi-mita, Tama-ku, Kawasaki 214-8571, Japan}
\email{naomatsu@meiji.ac.jp}

\thanks{The first author was partially supported by JSPS Grant-in-Aid for Young Scientists 20K14299. 
The second author was partially supported by JSPS Grant-in-Aid for Scientific Research (C) 21K03211. }

\begin{document}

\maketitle

\setlength{\baselineskip} {15.3pt}

\begin{abstract}
The Ulrich ideals in the semigroup rings $k[[t^5, t^{11}]]$ and $k[[t^5,t^6,t^9]]$ are determined, by describing the normal forms of systems of generators, where $k[[t]]$ denotes the formal power series ring over a field $k$. 
\end{abstract}



\section{Introduction}\label{intro}

This research is one of the attempts of determining the Ulrich ideals in the semigroup rings of numerical semigroups, and the present one is a succession of \cite{e=4}, where the authors started a systematic study of the ubiquity of Ulrich ideals inside semigroup rings. They succeeded in providing the normal forms of systems of generators of the Ulrich ideals, particularly in the case where the multiplicity of the semigroups is at most three. They pinpointed also the set $\calX_{k[[t^4,t^{13}]]}$ of Ulrich ideals in the ring $k[[H]]=k[[t^4,t^{13}]]$, the semigroup ring of the numerical semigroup $H=\left<4,13 \right>$ generated by $4, 13$, where $k[[t]]$ denotes the formal power series ring over a field $k$. The accurate arguments in \cite{e=4} have brought a new point of view, not only to the study of Ulrich ideals but also to further problems about numerical semigroups. The present purpose is, based on the technique developed by \cite{e=4}, to explore mainly the case of $k[[t^5,t^{11}]]$, which has been predicted in \cite{e=4} but left for another occasion. 


\section{Brief review on Ulrich ideals and preliminaries}
The notion of Ulrich ideal is one of the modifications of {\it stable} maximal ideal introduced in 1971 by J. Lipman \cite{L}. The present modification was formulated by \cite{GOTWY} in 2014. 

Let $(A, \m) $ be a Cohen-Macaulay local ring with $\dim A=d \ge 0$, and $I$ an $\m$-primary ideal of $A$. We throughout assume that $I$ contains a parameter ideal $Q$ of $A$ as a reduction.

\begin{defn} (\cite[Definition 1.1]{GOTWY})\label{def}
We say that $I$ is  an {\it Ulrich} ideal of $A$, if the following conditions are satisfied.
\begin{enumerate}[$(1)$]
\item $I \ne Q$,  $I^2=QI$, and
\item $I/I^2$ is a free $A/I$-module.
\end{enumerate}
\end{defn}

\noindent
Notice that Condition $(1)$ of Definition \ref{def} is satisfied if and only if the associated graded ring $\gr_I(A) = \bigoplus_{n\ge 0} I^n/I^{n+1}$ is a Cohen-Macaulay ring with $\rma(\gr_I(A))=1-d$, where $\rma(\gr_I(A))$ denotes the $\rma$-invariant of $\gr_I(A)$ (\cite[Definition 3.1.4]{GW}). Therefore, Condition $(1)$ is independent of the choice of reductions $Q$ of $I$. When $I=\m$, Condition $(2)$ is automatically satisfied, while Condition $(1)$ is equivalent to saying that $A$ is not regular but of minimal multiplicity.

Let $I$ be an $\m$-primary ideal of $A$ and assume that $I^2=QI$. Then, since $Q/QI$ is a free $A/I$-module of rank $d$, the exact sequence
$$
0 \to Q/QI \to I/I^2 \to I/Q \to 0
$$
of $A/I$-modules shows that $I/I^2$ is a free $A/I$-module if and only if so is $I/Q$. Therefore, provided $I$ is minimally generated by $d+1$ elements, the latter condition is equivalent to saying that $I/Q \cong A/I$ as an $A/I$-module. If $I$ is an Ulrich ideal, then by \cite{GOTWY, GTT2} we get the equality 
$$
(\mu_A(I)-d)\cdot\rmr(A/I)= \rmr(A),
$$
where $\mu_A(I)$ (resp. $\rmr(*)$) denotes the number of generators of $I$ (resp. the Cohen-Macaulay type). Therefore,
$d+1 \le \mu_A(I) \le d + \rmr(A)$,
so that when $A$ is a Gorenstein ring, every Ulrich ideal $I$ is generated by $d+1$ elements (if it exists).  As is shown in \cite{GOTWY, GTT2}, all the Ulrich ideals with $\mu_A(I)=d+1$ possess finite G-dimension, and their minimal free resolutions have a restricted form, so that they are eventually periodic of period one.

For instance, assume that $\dim A=1$ and let $I$ be an Ulrich ideal of $A$. Therefore, $I$ is an $\m$-primary ideal of $A$, and $I^2 = aI$ for some $a \in I$, such that $I \ne (a)$ but $I/(a)$ is a free $A/I$-module.   Assume that $I$ is minimally generated by {\it two} elements, say $I=(a, b)$ with $b \in I$, and write $b^2 = ac$ for some $c \in I$. We then have, since $I/(a) \cong A/I$, that $(a):_Ab = I$, and the minimal free resolution of $I$ has the following form
$$
\ \ \cdots \longrightarrow A^{\oplus 2} \overset{
\begin{pmatrix}
-b & -c\\
a & b
\end{pmatrix}}{\longrightarrow}
A^{\oplus 2} \overset{
\begin{pmatrix}
-b & -c\\
a & b
\end{pmatrix}}{\longrightarrow}A^{\oplus 2} \overset{\begin{pmatrix}
a & b
\end{pmatrix}}{\longrightarrow} I \longrightarrow 0
$$
(\cite[Example 7.3]{GOTWY}). In particular,  $I$ is a {\it totally reflexive} $A$-module, that is $I$ is a  reflexive $A$-module, $\Ext_A^p(I, A) =(0)$, and $\Ext_A^p(\Hom_A(I, A), A) = (0)$ for all $p >0$ (\cite[Proposition 4.6]{GIT}). We clearly have that $I =J$, once $\operatorname{Syz}_A^i(I) \cong \operatorname{Syz}_A^i(J)$ for some $i \ge 0$, provided $I,J$ are Ulrich ideals of $A$.

It seems that behind the behavior of Ulrich ideals and their existence also, there is hidden some ample information about the structure of the base rings. For example, if $A$ has finite Cohen-Macaulay representation type, then $A$ contains only finitely many Ulrich ideals  (\cite{GOTWY}). In a one-dimensional non-Gorenstein almost Gorenstein local ring, the only possible Ulrich ideal is the maximal ideal (\cite[Theorem 2.14]{GTT2}). In \cite{GIT}  the authors explored the ubiquity of Ulrich ideals in a {\it $2$-AGL} rings (one of the generalizations of Gorenstein local rings of dimension one), and showed that the existence of two-generated Ulrich ideals provides a rather strong restriction on the structure of the base local rings (\cite[Theorem 4.7]{GIT}). The motivation for our research comes from these observations.

Let us summarize a few results which we later need in this paper. Throughout, let $k$ be a field and let $V = k[[t]]$ denote the formal power series ring over $k$. Let $A$ be a $k$-subalgebra of $V$. Then, following \cite{EGMY}, we say that $A$ is a {\it core}  of $V$, if $t^c V \subseteq A$ for some $c \gg 0$. The semigroup rings $k[[H]]=k[[t^{a_i} \mid 1 \le i \le \ell]]$ of numerical semigroups $H=\left<a_1, a_2, \ldots, a_\ell \right>$ are typical examples of cores of $V$. If $I$ is an Ulrich ideal in the semigroup ring $A$ of a numerical semigroup, the blowing-up ring $A^I = \bigcup_{n \ge 0}[I^n:I^n]$ of $A$ with respect to $I$ is again a core of $V$, which is, however, not necessarily a semigroup ring (see, e.g., Proposition  \ref{prop2}). We would like to refer the readers to \cite{RG} for general results on numerical semigroups.

Let $I$ be a fixed two-generated  Ulrich ideal of a core $A$ of $V$. Let $f, g \in I$ such that $I = (f, g)$ and $I^2=fI$. We consider the $A$-subalgebra $$A^I=\bigcup_{n \ge 0}[I^n : I^n]$$
of $V$, where the colon $$I^n:I^n =\{x \in \operatorname{Q}(A) \mid xI^n \subseteq I^n \}$$ is considered inside the quotient field of $A$. We then have $A^I = I:I$ since $I^{n+1}=f^{n}I$ for all $n \ge 0$, so that $A^I = f^{-1}I =A + A{\cdot}\frac{g}{f}$.  We set 
$$a = \rmo(f), \ \  b = \rmo(g), \ \ \text{and}\ \ c = \rmc(H)$$
where $\rmc(H)$ denotes the conductor of $H$. Notice that $a$ is an invariant of $I$, since $IV=fV=t^aV$. We set $v(A) = \{\rmo(f) \mid f \in A\}$, where $\rmo(*)$ denotes the valuation (or the order function) of $V$. With this setting and notation, we have the following, which plays a key role throughout this paper.

\begin{lem}[{\cite[Lemma 2.3, Theorem 2.7]{e=4}}]\label{key}
Let $A$ be a core of $V=k[[t]]$
 and let $H=v(A)$. Let $I$ be an Ulrich ideal in $A$ and assume that $I$ is minimally generated by two elements.  Then one can choose elements $f, g \in I$ so that the following conditions are satisfied, where $a = \rmo(f)$, $b=\rmo(g)$, $c = \rmc(H)$, and $\fkc = A:V$.
\begin{enumerate}[$(1)$]
\item $I=(f, g)$ and $I^2= fI$. 
\item $a, b \in H$ and $0 < a < b < a + c$.
\item $b-a \not\in H$, $2b-a \in H$, and $a = 2\cdot \ell_A(A/I)$.
\item $t^{c-(b-a)}V \cap A \subseteq I$ and $\fkc \subseteq I$.
\item If $a \ge c$, then $\rme(A) = 2$ and $I=\fkc$.
\end{enumerate}
\end{lem}


\section{The case where $H = \left<5,6,9 \right>$}
Let us begin with the following.

Let $H=\left<5, 6, 9 \right>$ 
$$
 \ytableausetup{mathmode, boxsize=2em}
 \begin{ytableau}
 *(lightgray) 0 & 1 &  2 & 3 & 4\\
 *(lightgray) 5 &  *(lightgray)6 & 7 & 8 &  *(lightgray)9 \\
 *(lightgray) 10 &  *(lightgray) 11 &  *(lightgray)12 & 13 &  *(lightgray)14 \\ 
 *(lightgray) 15 & *(lightgray)16 &  *(lightgray)17 &  *(lightgray)18 &  *(lightgray)19\\ 
 *(lightgray) \vdots   &  *(lightgray)  \vdots &  *(lightgray) \vdots &  *(lightgray)\vdots &  *(lightgray)\vdots\\
\end{ytableau}
\vspace{0.5em} 
$$
and set $A = k[[t^5,t^{6}, t^{9}]]$. Then, because $A$ is a Gorenstein ring, every possible Ulrich ideal must be two-generated and we have the following.

\begin{prop}\label{569}
$\calX_A=\{(t^6 +\alpha t^{10}, t^9 + \beta t^{10}) \mid \alpha, \beta \in k, 2\beta = 0\}$.
The elements $\alpha, \beta \in k$ in the expression $I =(t^6 +\alpha t^{10}, t^9 + \beta t^{10})$ are uniquely determined for each $I \in \calX_A$.
\end{prop}

\begin{proof}
Let $I \in \calX_A$ and choose elements $f,g \in I$ so that all the conditions stated in Lemma \ref{key} are satisfied. We maintain the notation given in Lemma \ref{key}. After some routine works, similarly as in the proof of \cite[Example 2.8, Proposition 4.9 (1), (2)]{e=4} we attain the unique possible pair $(a,b)= (6,9)$. We write 
$$f = t^6 + \alpha_1 t^{9} + \alpha_2 t^{10} + \alpha_3 t^{11} + \alpha_4 t^{12} + \rho,\ \ g=t^9 + \beta_1 t^{10} + \beta_2 t^{11} + \beta_3 t^{12} + \eta$$
where $\alpha_i, \beta_j \in k$ and $\rho, \eta \in \fkc=t^{14}V$. We then have after some elementary transform on $f, g$ that
$$I = (f,g)=(f,g)+\fkc = (t^6 + \alpha_2t^{10}, t^9 + \beta_1 t^{10})+(t^{14}, t^{15}, t^{16}, t^{17}, t^{18})$$
 (see the proof of \cite[Corollary 3.8]{e=4}), which shows $$I = (t^6+\alpha_2t^{10}, t^9 + \beta_1t^{10}),$$ since $\mu_A(I) = 2$ and $t^9 + \beta_1 t^{10} \not\in (t^6+\alpha_2 t^{10}, t^{14},t^{15}, t^{16}, t^{17}, t^{18})$. Thus we may assume
$$f = t^6 + \alpha t^{10}, \ \ g = t^{9}+ \beta t^{10}$$
where $\alpha, \beta \in k$. Let $\xi= \frac{g}{f}$. We then have $$\xi = t^3 + \beta t^4 - \alpha t^7 - \alpha \beta t^8 + \alpha^2 t^{11} + \cdots.$$ Consequently
$$\frac{g^2}{f} =g\xi = t^{12} + 2\beta t^{13} + \eta$$
where $\eta \in V$ with $\rmo(\eta) \ge 14$. Since $\frac{g^2}{f} \in I \subseteq A$ but $13 \not\in H$, this implies $2\beta =0$.

Conversely, let $f = t^6 + \alpha t^{10}, \ g = t^{9}+ \beta t^{10}$
where $\alpha, \beta \in k$ such that $2\beta = 0$, and set $I = (f,g)$. We must show that $I \in \calX_A$. Let $L = v(I)$. Then, $14, 15, 16, 17, 18 \in L$, so that $\fkc=t^{14}V \subseteq I$. Hence, the vanishing 
$f \equiv 0, g \equiv 0~\mod~I$ forces the $k$-space $A/I$ to be spanned by the images of the monomials $1, t^5, t^{10}$, whence $\ell_A(A/I) \le 3$. Therefore, the epimorphism 
$$\varphi: A/I \to I/(f),\ \ \varphi(1~\mod~I)=g~\mod~I$$ of $A$-modules is an isomorphism, since $\ell_A(A/(f))=6$. Thus, $A/I \cong I/(f)$ as an $A$-module and the images of $1, t^5, t^{10}$ form a $k$-basis of $A/I$.  It now suffices to show $I^2 = fI$. To see it, notice that $t^{12} \in I$, since $t^{6}f = t^{12} + \alpha t^{16} \in I$. Therefore 
$$\frac{g^2}{f} \equiv t^{12} ~\mod~t^{14}V,$$ so that $\frac{g^2}{f} \in I$, because $t^{14}V \subseteq I$ and $t^{12} \in I$. Hence, $I^2=fI$, and  $I \in \calX_A$.

Let us check the second assertion. Let $f=t^6 +\alpha t^{10}, g=t^9 + \beta t^{10}$ and $f_1=t^6 +\alpha_1 t^{10}, g_1=t^9 + \beta_1 t^{10}$ where $\alpha, \beta, \alpha_1, \beta_1 \in k$, and assume that $I= (f,g)=(f_1,g_1)$. Then, considering $f-f_1$ and $g-g_1$, we get $(\alpha -\alpha_1)t^{10}, (\beta -\beta_1)t^{10}\in I$. Therefore, $\alpha = \alpha_1$ and $\beta = \beta_1$, because the images of $1, t^5, t^{10}$ form a $k$-basis of $A/I$ as we showed above. 
\end{proof}


\section{The case where $H =\left<5, 11 \right>$ and $(a,b)=(10, 27)$}\label{section2}

Let $H=\left<5, 11 \right>$ 
$$
 \ytableausetup{mathmode, boxsize=2em}
 \begin{ytableau}
 *(lightgray) 0 & 1 &  2 & 3 & 4\\
 *(lightgray) 5 & 6 & 7 & 8 & 9 \\
 *(lightgray) 10 &  *(lightgray) 11 & 12 & 13 & 14 \\ 
 *(lightgray) 15 &  *(lightgray) 16 & 17 & 18 & 19\\ 
 *(lightgray) 20   &   *(lightgray) 21 & *(lightgray) 22 & 23 & 24\\
 *(lightgray) 25   &  *(lightgray) 26 &  *(lightgray) 27 & 28 & 29\\ 
 *(lightgray) 30   &   *(lightgray) 31 & *(lightgray) 32 &  *(lightgray) 33 & 34 \\ 
 *(lightgray) 35   &  *(lightgray) 36  &  *(lightgray) 37 &   *(lightgray) 38 & 39\\
 *(lightgray) 40  &   *(lightgray)41 &  *(lightgray)42 &  *(lightgray)43 &  *(lightgray)44\\
 *(lightgray) \vdots   &  *(lightgray)  \vdots &  *(lightgray) \vdots &  *(lightgray)\vdots &  *(lightgray)\vdots\\
\end{ytableau} 
$$
and set $A = k[[t^5,t^{11}]]$. We now want to determine all the Ulrich ideals in $A$.
Let $I \in \calX_A$ and choose elements $f,g \in I$ so that the conditions stated in Lemma \ref{key} are satisfied. We maintain the notation given in Lemma \ref{key}. After some routine works, similarly as in the proof of \cite[Example 2.8, Proposition 4.9 (1), (2)]{e=4}, we are able to restrict the possible pairs $(a,b)$ within $$(10,16), \ (10, 27),\  (20, 26),$$ and eventually see the following.

\begin{prop}\label{lemma1}
$(a,b) \ne (10, 16)$.
\end{prop}

\begin{proof}
Assume $(a,b)=(10, 16)$ and consider the table:
\vspace{1em}
$$
\begin{tabular}{|c|c|}
\hline
$a$ & $10$\\ \hline 
$b-a$ & $6$\\ \hline 
$b$ & $16$\\ \hline
$2b-a$ & $22$\\ \hline
$40-(b-a)$ & $34$\\ \hline
\end{tabular} \vspace{1em}
$$
We set $J = (t^{35}, t^{36}, t^{37}, t^{38})$. Then, $J \subseteq I$ by Lemma \ref{key} (4), since $40-(b-a)=34$.  Therefore, some elementary transform on $f,g$ similar to the one given in the proof of \cite[Corollary 3.8]{e=4} implies 
$$
I=(f,g)+J=(t^{10}+\alpha_1 t^{11}+\alpha_2t^{22}+\alpha_3t^{33},t^{16} +\beta_1 t^{22} + \beta_2t^{33})+J$$
where $\alpha_i, \beta_j \in k$. Hence, we may assume that
$$
f=t^{10}+\alpha_1 t^{11}+\alpha_2t^{22}+\alpha_3t^{33}, \ \ g=t^{16} +\beta_1 t^{22} + \beta_2t^{33},
$$
because $\mu_A(I)=2$ and $6 \not\in H$.
Let us write $$\frac{g^2}{f} = t^{22} + \gamma_1 t^{25} + \gamma_2 t^{26} + \gamma_3 t^{27}+ \gamma_4 t^{30} + \gamma_5 t^{31} + \gamma_6 t^{32} + \gamma_7 t^{33} + \delta$$
where $\gamma_i \in k$ and $\delta \in J$. Then, since $\rmo(f)=10$ and $\rmo(g)=16$, we readily have
$$t^{22} + \gamma t^{33} =t^{22}(1 + \gamma t^{11}) \in I$$
where $\gamma=\gamma_7$. Therefore, $t^{22} \in I$, so that $$I = (f,g, t^{22}) = (t^{10}+\alpha_1t^{11}, t^{16}, t^{22})=(t^{10}+\alpha_1t^{11}, t^{16}),$$
since $t^{16} \not\in (t^{10}+\alpha_1t^{11}, t^{22})$.
This enables us, from the beginning, to assume that $f = t^{10}+\alpha_1 t^{11}$ and $g=t^{16}$. Since $t^{22} \in I$, we then have
$$
t^{22}=(t^{10}+\alpha_1 t^{11})\varphi + t^{16}\psi
$$
with $\varphi, \psi \in A$, which shows $\varphi_{11} \alpha_1=1$ (here $\varphi_{11} \in k$ denotes the coefficient of $t^{11}$ in $\varphi$). Consequently, $\alpha_1 \ne 0$, and therefore, the equation
$$ 
t^{32}=(t^{10}+\alpha_{1} t^{11})t^{22}-\alpha_{1} t^{33}
$$
implies $t^{33} \in fI+(g^2)=fI$, whence $23 \in H$. This is impossible. 
\end{proof}

\begin{prop}\label{prop1}
\begin{enumerate}[$(1)$]
\item If $(a,b)=(10,27)$, then $$I=(t^{10} +\alpha_1 t^{11}+\alpha_2 t^{16} +\alpha_3 t^{22}, t^{27})$$ for some $\alpha_i \in k$ such that $\alpha_1 \ne0$.
\item If $(a,b)= (20,26)$, then $$I=(t^{20} +\alpha_1t^{21}+\alpha_2 t^{22} +\alpha_3 t^{27}+\alpha_4 t^{33}, t^{26} +\beta_1 t^{27} + \beta_2t^{33})$$ and $t^{32} + \varepsilon  t^{33} \in I$ for some $\alpha_i, \beta_j, \varepsilon \in k$.
\end{enumerate}
\end{prop}

\begin{proof}
(1)~Since $40-(b-a)=23$, we get $(t^{30}, t^{31}, t^{32}, t^{33}, t^{44}) \subseteq I$, and we can assume
$$
f=t^{10} +\alpha_1 t^{11}+\alpha_2 t^{16} +\alpha_3 t^{22}\ \ \text{and}\ \ g=t^{27}
$$
where $\alpha_i \in k$. The assertion that $\alpha_1 \ne 0$ follows from the fact that $t^{33} \in I$.

(2)~We have $(t^{35}, t^{36}, t^{37}, t^{38}, t^{44})  \subseteq I$, since $c-(b-a)=34$, while  $$t^{32} + \varepsilon_{1} t^{33} + \varepsilon_{2} t^{35} + \varepsilon_{3} t^{36} + \varepsilon_{4}t^{37}+ \varepsilon_{5} t^{38} + \rho \in I$$ for some $\varepsilon_i \in k$ and $\rho \in t^{40}V$, since $2b-a=32$. Therefore, $t^{32} + \varepsilon  t^{33} \in I$ for $\varepsilon=\varepsilon_1$, whence  we can choose $f,g$ so that 
$$
f=t^{20} +\alpha_1t^{21}+\alpha_2 t^{22} +\alpha_3 t^{27}+\alpha_4 t^{33}\ \ \text{and}\ \ g=t^{26} +\beta_1 t^{27} + \beta_2t^{33}
$$
where $\alpha_i, \beta_j \in k$.
\end{proof}

Proposition \ref{prop1} (1) gives the normal form of systems of generators of Ulrich ideals $I$ possessing $(a,b)= (10,27)$, as the following result shows.

\begin{thm}\label{thm1}
Let $f= t^{10} +\alpha_1 t^{11}+\alpha_2 t^{16} +\alpha_3 t^{22}$ and $g=t^{27}$ with $\alpha_i \in k$ such that $\alpha_1 \ne 0$. Then $I=(f,g)$ is an Ulrich ideal of $A$. The elements $\{\alpha_i\}_{i = 1,2,3}$ in the expression of $I$ are uniquely determined by $I$. 
\end{thm}

\begin{proof}
We set $I=(f,g)$ and $L =v(I)$. Then, $40, 41, 42, 43, 44 \in L$. In fact, since $10, 27 \in L$, we readily have $40,41, 42, 43 \in L$. Since $t^{32} =t^5g$, $t^{38} = t^{11}g$, and $$t^{22}f = t^{32} + \alpha_1 t^{33}+ \alpha_2 t^{38} + \alpha_3 t^{44},$$ we get $t^{33}+ \frac{\alpha_3}{\alpha_1 } t^{44} \in I$  (remember $\alpha_1 \ne 0$), so that $33 \in L$. Therefore, $40, 41, 42, 43, 44 \in L$, which implies that $\fkc=t^{40}V \subseteq I$. Therefore, $t^{33} \in I$. Since $\frac{g^2}{f} \in \fkc$ (notice that $\rmo(\frac{g^2}{f})= 44 > \rmc(H)$), we have $\frac{g^2}{f} \in I$, whence $I^2 = fI$. Hence, to see that $I$ is an Ulrich ideal of $A$, it suffices to show that $I/(f)$ is a free $A/I$-module.

We now notice that $k$-space $A/I$ is spanned by the images of the monomials $\{t^q\}$ $(0 \le q \le 38, q \ne 33)$. Among them, there are relations
$$t^{10} +\alpha_1 t^{11}+\alpha_2 t^{16} +\alpha_3 t^{22} \equiv 0~\mod~I\ \ \text{and}\ \ t^{27} \equiv 0~\mod~I$$
induced by the vanishing of $f,g$ $\mod~I$, which implies that the $k$-space $A/I$ is actually spanned by the images of the following five monomials
$$1, t^5, t^{11}, t^{16}, t^{22}$$
so that $\ell_A(A/I) \le 5$, which is enough to guarantees that $I/(f)$ is a free $A/I$-module. In fact, the epimorphism
$$\varphi : A/I \to I/(f), \ \ \varphi(1\mod~I)=g~\mod~I$$
of $A$-modules must be an isomorphism, since $$\ell_A(I/(f))=\ell_A(A/(f))-\ell_A(A/I) =10 -\ell_A(A/I) \ge 5.$$
 Thus, $I$ is an Ulich ideal of $A$.

The second assertion follows from the fact that the images of $1, t^5, t^{11}, t^{16}, t^{22}$ form a $k$-basis of $A/I$. 
\end{proof}


\section{The case where $H=\left<5, 11 \right>, (a,b)=(20,26)$, and $\operatorname{ch}(k)= 2$}
We shall study the case where $(a,b)=(20, 26)$. Our goal is the following.

\begin{thm}\label{thm1}\label{main} Let $\calY_A$ denote the set of Ulrich ideals of $A$ which possess the data $(a,b) = (20, 26)$. Then 
$$
\calY_A=\left\{(t^{20} + 3\varepsilon  t^{21} + \varepsilon^2 t^{22} + (2\tau - 55 \varepsilon^7)t^{27} + \delta t^{33}, t^{26} + 2\varepsilon t^{27} + \tau t^{33}) \mid \delta, \varepsilon, \tau \in k, \varepsilon \ne 0 \right\}.
$$ The elements $\delta, \varepsilon, \tau \in k$ in the expression of $I \in \calY_A$ are uniquely determined for each $I$. 
\end{thm}

Let us begin with the following.

\begin{lem}\label{2.1}
Let $\alpha, \beta, \gamma, \delta, \varepsilon, \tau \in k$ and assume that $\varepsilon \ne \alpha, \beta \ne 0$, and $\beta \ne \varepsilon (\alpha - \varepsilon)$. 
Let $f=t^{20} +\alpha t^{21}+\beta t^{22} +\gamma t^{27}+\delta t^{33}$ and $g= t^{26} +\varepsilon t^{27} + \tau t^{33}$. We set $J=(f,g)$. Then the following assertions hold true.
\begin{enumerate}[$(1)$]
\item $\fkc =t^{40}V  \subseteq J$ and $A/J \cong J/(f)$ as an $A$-module.
\item $J$ is an Ulrich ideal of $A$ if and only if $g^2 \in fJ$.
\end{enumerate}
\end{lem}

\begin{proof}
We set $L = v(I)$. First of all, we will show that $45, 46, 47, 48, 49 \in L$. It suffices to see $38 \in L$, since $45,46,47 \in L$. Consider
\begin{eqnarray*}
t^{16}f&=&t^{36} +\alpha t^{37}+\beta t^{38} +\gamma t^{43}+\delta t^{49} \\ 
t^{10}g&=& t^{36} +\varepsilon t^{37} + \tau t^{43}\\
t^{11}g&=&t^{37} +\varepsilon t^{38} + \tau t^{44}
\vspace{-2em}
\end{eqnarray*}
and we have 
$$
(*) \ \ \ \ \ t^{10}g -t^{16}f= \left(\varepsilon -\alpha \right)t^{37}-\beta t^{38}+(\tau -\gamma)t^{43} - \delta t^{49} \in I
$$
so that 
$$
(**) \ \ \ \ \ (\varepsilon-\alpha){\cdot} t^{11}g -(t^{10}g -t^{16}f)=(\varepsilon(\varepsilon -\alpha)+\beta)t^{38} + (\gamma - \tau)t^{43} + (\varepsilon - \alpha)t^{44} + \delta t^{49} \in I.$$
Consequently, $38 \in L$, since $\varepsilon(\varepsilon -\alpha)+\beta \ne 0$. Hence, $45,46, 47, 48, 49 \in L$, and $t^{45}V \subseteq I$. We now notice that
\begin{eqnarray*}
\hspace{1.5em}
(***1) \ \ \ \ \ t^{20}f&=&t^{40} +\alpha t^{41}+\beta t^{42} +\gamma t^{47}+\delta t^{53}\\
(***2) \ \ \ \ \ t^{15}g&=&t^{41} +\varepsilon t^{42} + \tau t^{48}\\
 \vspace{-2cm}
\end{eqnarray*}
Then, since
$$
(\varepsilon(\varepsilon-\alpha)+\beta)t^{43} + (\gamma - \tau)t^{48} + (\varepsilon - \alpha)t^{49} + \delta t^{54} \in I
$$
by $(**)$, 
we have $(\varepsilon(\varepsilon-\alpha)+\beta)t^{43} \in I$, so that $t^{43} \in I$. Consequently, since 
$$
t^5(t^{10}g -t^{16}f)= \left(\varepsilon -\alpha \right)t^{42}-\beta t^{43}+(\tau -\gamma)t^{48} - \delta t^{54} \in I
$$
by $(*)$, 
we get 
$$
(\varepsilon - \alpha)t^{42} \in I,
$$
whence $t^{42} \in I$. Therefore $t^{41}, t^{40} \in I$ by $(***2)$ and $(***1)$. Similarly, considering $t^{22}f -t^{16}g$, we have 
$$(\varepsilon -\alpha) t^{43} - \beta t^{44} \in I,$$ so that $t^{44} \in I$. Therefore, we have $\fkc = t^{40}V \subseteq I$. On the other hand, $t^{38} \in I$ by $(**)$, whence we get $t^{37} \in I$, considering $t^{11}g$. Hence $t^{36}, t^{35} \in I$, which come from the fact $t^{16}g, t^{15}f \in I$. Therefore, $$t^{35}, t^{36}, t^{37}, t^{38} \in I,$$ which guarantees that the $k$-space $A/I$ is spanned by the images of the monomials
$$1, \ t^5, \ t^{10},\  t^{11}, \ t^{15}, \ t^{16},\  t^{20},\  t^{21}, \ t^{22}, \ t^{25}, \ t^{26},\ t^{27}, \ t^{30},\ t^{31},\ t^{32}, \  t^{33}$$
$$
 \ytableausetup{mathmode, boxsize=2em}
 \begin{ytableau}
 *(lightgray) 0 & 1 &  2 & 3 & 4\\
 *(lightgray) 5 & 6 & 7 & 8 & 9 \\
 *(lightgray) 10 &  *(lightgray) 11 & 12 & 13 & 14 \\ 
 *(lightgray) 15 &  *(lightgray) 16 & 17 & 18 & 19\\ 
 *(lightgray) 20   &   *(lightgray) 21 & *(lightgray) 22 & 23 & 24\\
 *(lightgray) 25   &  *(lightgray) 26 &  *(lightgray) 27 & 28 & 29\\ 
 *(lightgray) 30   &   *(lightgray) 31 & *(lightgray) 32 &  *(lightgray) 33 & 34 \\ 
\end{ytableau} 
$$
and among them we have the relations 
$$
f\equiv 0,\  t^{15}f \equiv 0,\  t^{10}f \equiv 0,\  g \equiv 0, \ t^5g \equiv 0~\mod~I,
$$
so that $A/I$ is actually spanned by the images of the following eleven monomials
$$1, \ t^5, \ t^{10},\  t^{11}, \ t^{15}, \ t^{16},\  t^{21}, \ t^{22}, \ t^{27}, t^{32},\  t^{33}.$$
We however one more relation $$t^5g-t^{11}f \equiv (\varepsilon -\alpha )t^{32} -\beta t^{33} \equiv 0~\mod~I,$$
so that $\ell_A(A/I) \le 10$. Therefore, because $\ell_A(A/(f))=20$, the epimorphim $$\varphi : A/I \to I/(f), \ \ \varphi(1~\mod~I) = g~\mod~I$$  of $A$-modules is an isomorphism, and thus, $A/I \cong I/(f)$ as an $A$-module.

The second assertion is clear, since $(f)$ is a reduction of $I$ (notice that $IV=fV$).
\end{proof}

Let $I$ be an Ulrich ideal of $A=k[[t^5,t^{11}]]$ and suppose that $(a,b)=(20, 26)$. Thanks to Proposition \ref{prop1} (2), we may assume that 
$$
f=t^{20} +\alpha t^{21}+\beta t^{22} +\gamma t^{27}+\delta t^{33}\ \ \text{and}\ \ g= t^{26} +\varepsilon t^{27} + \tau t^{33}
$$
for some $\alpha, \beta, \gamma, \delta, \varepsilon, \tau \in k$, where we shall fix the present notation which is different from the notation used in Proposition \ref{prop1} (2), in order to avoid possible confusion about the indices. We set $\xi = \frac{g}{f}$, $B=A^I$, and $H_1 = v(B)$. Then $\rmo(\xi)=6$, $B=k[[t^5,t^{11},\xi]]$, and $B=f^{-1}I = A+A\xi$. The numerical semigroup $H_1$ is symmetric, since $B$ is a Gorenstein ring. Notice that $\xi \not\in \m B$ but $\xi^2 \in \m B$ (\cite[Lemma 3.2]{e=4}), where $\m$ denotes the maximal ideal of $A$. 

We then have the following.

\begin{prop}\label{prop2}
$H_1= \left<5,6 \right>$, $B = k[[t^5, \xi]]$, and $t^6 \not\in B$. In particular, $B$ is not the semigroup ring for any numerical semigroup.
\end{prop}

\begin{proof}
Since $\m_BV=t^5V$, we have $\left<5,6 \right> \subseteq H_1 \subseteq \left<5,6,7,8,9 \right>$, and $\m_B^2 =t^5 \m_B$, where $\m_B$ denotes the maximal ideal of $B$. Hence, $7,8,9 \not\in H_1$. In fact, let $q \in \{7,8,9\}$ and assume $q \in H_1$. Choose an element $\eta \in B$ so that $\rmo(\eta)=q$. We then have $\eta =t^5\varphi + t^{11}\psi + \xi \rho$ for some $\varphi, \psi, \rho \in B$, where $\varphi, \rho \in \m_B$ because $\rmo(\eta) =q \ge 7$. Hence, $\rmo(t^5\varphi) \ge 10$ and $\rmo(\xi \rho) \ge 11$, which forces $q=\rmo(\eta) \ge 10$. This is absurd. Thus, $7,8,9 \not\in H_1$.

Since $\xi^2 \in \m B$, we have $$\xi^2 = t^5\varphi+ t^6{\cdot}t^5\psi=t^5(\varphi + t^6\psi)$$ with $\varphi, \psi \in B$. Hence, if $t^6 \in B$, then $\rmo(\varphi + t^6\psi)= 7$, so that $7 \in H_1$, which is impossible. Hence $t^6 \not\in B$, and $B$ is not the semigroup ring for any numerical semigroup.

We set $C = k[[t^5, \xi]]$. Then $C \subseteq B$ and $\left<5, 6 \right> \subseteq v(C)$. Therefore, to see $C=B$, it suffices to show $H_1= \left<5,6 \right>$.

$\left<5,6\right>$:
\vspace{0.5em}
$$
 \ytableausetup{mathmode, boxsize=2em}
 \begin{ytableau}
 *(lightgray) 0 & 1 &  2 & 3 & 4\\
 *(lightgray) 5 & *(lightgray)6 & 7 & 8 & 9 \\
 *(lightgray) 10 &  *(lightgray) 11 & *(lightgray)12 & 13 & 14 \\ 
 *(lightgray) 15 &  *(lightgray) 16 & *(lightgray)17 &*(lightgray) 18 & 19\\ 
 *(lightgray) 20   &   *(lightgray) 21 & *(lightgray) 22 & *(lightgray)23 & *(lightgray)24\\
\end{ytableau} 
$$
\vspace{-0.5em}

Assume that $H_1 \supsetneq \left<5,6 \right>$. Then, $19 \in H_1$, since the $C$-submodule $V/C$ of $\rmQ(C)/C$ contains a unique socle generated by the image of $t^{19}$ (here $\rmQ(C)$ denotes the quotient field of $C$) . We set $D = k[[t^5,t^6,t^{19}]]$.

$\left<5,6,19 \right>$:
\vspace{0.5em}
$$
 \ytableausetup{mathmode, boxsize=2em}
 \begin{ytableau}
 *(lightgray) 0 & 1 &  2 & 3 & 4\\
 *(lightgray) 5 & *(lightgray)6 & 7 & 8 & 9 \\
 *(lightgray) 10 &  *(lightgray) 11 & *(lightgray)12 & 13 & 14 \\ 
 *(lightgray) 15 &  *(lightgray) 16 & *(lightgray)17 &*(lightgray) 18 & *(lightgray) 19\\ 
\end{ytableau} 
$$
\vspace{-0.5em}

\noindent
Then, because $\left<5,6,19 \right>$ is not symmetric, we have $H_1 \supsetneq \left<5,6,19 \right>$, whence $13 \in H_1$ or $14 \in H_1$, because the socle of the $D$-module $V/D$ is spanned by the images of $t^{13}$ and $t^{14}$. We claim $13 \not\in H_1$. In fact, suppose $13 \in H_1$. Then,  $H_1 \supsetneq \left<5,6,13 \right>$ since $\left<5,6,13 \right>$ is not symmetric, so that $14 \in H_1$, because the socle of the $k[[t^5,t^6,t^{13}]]$-module $V/k[[t^5,t^6,t^{13}]]$ is spanned by the images of $t^7$ and $t^{14}$ but $7 \notin H_1$.

$\left<5,6,13 \right>$:
\vspace{0.5em}
$$
 \ytableausetup{mathmode, boxsize=2em}
 \begin{ytableau}
 *(lightgray) 0 & 1 &  2 & 3 & 4\\
 *(lightgray) 5 & *(lightgray)6 & 7 & 8 & 9 \\
 *(lightgray) 10 &  *(lightgray) 11 & *(lightgray)12 & *(lightgray)13 & 14 \\ 
 *(lightgray) 15 &  *(lightgray) 16 & *(lightgray)17 &*(lightgray) 18 & *(lightgray) 19\\ 
\end{ytableau} 
$$
\vspace{-0.5em}

\noindent
Therefore, $H_1 \supseteq \left<5,6,13, 14 \right>$, whence $H_1 = \left<5,6,13, 14 \right>$, because $7, 8, 9 \not\in H_1$.

$\left<5,6,13, 14 \right>$:
\vspace{0.5em}
$$
 \ytableausetup{mathmode, boxsize=2em}
 \begin{ytableau}
 *(lightgray) 0 & 1 &  2 & 3 & 4\\
 *(lightgray) 5 & *(lightgray)6 & 7 & 8 & 9 \\
 *(lightgray) 10 &  *(lightgray) 11 & *(lightgray)12 & *(lightgray)13 & *(lightgray)14 \\ 
 *(lightgray) 15 &  *(lightgray) 16 & *(lightgray)17 &*(lightgray) 18 & *(lightgray) 19\\ 
\end{ytableau} 
$$
\vspace{-0.5em}

\noindent
This is impossible, since $\left<5,6,13,14 \right>$ is not symmetric. Thus, $13 \not\in H_1$, so that $14 \in H_1$ and hence $H_1 = \left<5,6,14 \right>$, because $7,8,9,13 \not\in H_1$.

$\left<5,6,14 \right>$:
\vspace{0.5em}
$$
 \ytableausetup{mathmode, boxsize=2em}
 \begin{ytableau}
 *(lightgray) 0 & 1 &  2 & 3 & 4\\
 *(lightgray) 5 & *(lightgray)6 & 7 & 8 & 9 \\
 *(lightgray) 10 &  *(lightgray) 11 & *(lightgray)12 & 13 & *(lightgray)14 \\ 
 *(lightgray) 15 &  *(lightgray) 16 & *(lightgray)17 &*(lightgray) 18 & *(lightgray) 19\\ 
\end{ytableau} 
$$
\vspace{-0.5em}

\noindent
This is, however, still impossible, since $\left<5,6, 14 \right>$ is not symmetric. Thus, $H_1 =\left<5,6 \right>$ as claimed. 
\end{proof}

The following is the key in our argument.

\begin{prop}\label{lemma2} We have $\varepsilon \ne \alpha$, $3 \varepsilon = 2\alpha$, and $\beta = \alpha^2-2\varepsilon^2$. 
\end{prop}

\begin{proof}
Let $$\xi= t^6 + c_7t^7 + c_8t^8 + c_9t^9+\cdots$$ with $c_i \in k$. We then have $$c_7=\varepsilon - \alpha,\ \ c_8=\alpha^2-\varepsilon \alpha - \beta, \ \ c_9= 2\alpha\beta + \varepsilon \alpha^2 -\varepsilon \beta - \alpha^3,$$
since $g=f\xi$. Let us write $\xi^2 =t^5\varphi+t^{11}\psi$ with $\varphi, \psi \in B$ and let $$\varphi = \sum_{i=0}^\infty \varphi_it^i \ \ \text{and} \ \ \psi = \sum_{i=0}^\infty \psi_it^i$$ with $\varphi_i, \psi_i \in k$. Then, comparing the coefficients of $t^i$ ($5 \le i \le 14$) in both sides of the equation
$$(t^6 + c_7t^7 + c_8t^8 + c_9t^9+\cdots)^2 = t^5{\cdot}\sum_{i=0}^\infty \varphi_it^i+t^{11}{\cdot}\sum_{i=0}^\infty \psi_it^i,$$
we have $\varphi_i=0$ for all $ 0 \le i \le 5$ and
$$
\varphi_6+\psi_0=0, \ \ \varphi_7=1, \ \ \varphi_8=2(\varepsilon -\alpha),\ \ \varphi_9=\varepsilon^2 -4\varepsilon \alpha + 3\alpha^2 -2 \beta.
$$
Hence $$\varphi = \varphi_6t^6 + t^7 + 2(\varepsilon -\alpha)t^{8}+ (\varepsilon^2 -4\varepsilon \alpha + 3\alpha^2 -2 \beta)t^9 + \cdots.$$ 
We now compare $\varphi$ with
$$\varphi_6\xi = \varphi_6t^6+ \varphi_6(\varepsilon-\alpha)t^7+\varphi_6(\alpha^2-\varepsilon \alpha - \beta)t^8+\varphi_6(2\alpha\beta + \varepsilon \alpha^2 -\varepsilon \beta - \alpha^3)t^9 + \cdots$$
and considering the coefficients of $t^i$ ($7 \le i \le 9$) in the difference $\varphi - \varphi_6\xi$, we get 
$$
1=\varphi_6{\cdot}(\varepsilon - \alpha), 2(\varepsilon - \alpha) = \varphi_6(\alpha^2-\varepsilon \alpha -\beta), \varepsilon^2 -4\varepsilon \alpha + 3\alpha^2 -2 \beta =\varphi_6(2\alpha \beta + \varepsilon \alpha^2 - \varepsilon \beta -\alpha^3)
$$
because $7, 8,9 \not\in H_1=\left<5,6\right>$. The first two equalities imply that
$$\varepsilon \ne \alpha \ \ \text{and}\ \ \beta =-2\varepsilon^2+3\varepsilon \alpha - \alpha^2,$$
and because 
$$
\varepsilon^2 -4\varepsilon \alpha + 3\alpha^2 -2 \beta =\frac{1}{\varepsilon -\alpha}(2\alpha \beta + \varepsilon \alpha^2 - \varepsilon \beta -\alpha^3),
$$
we have
$$
3\varepsilon^3 - 8\varepsilon^2\alpha +7\varepsilon \alpha^2 -2\alpha^3=(\varepsilon - \alpha)^2(3\varepsilon -2\alpha)=0.
$$
Thus $$3\varepsilon=2\alpha\ \ \text{and}\ \ \beta =-2\varepsilon^2+3\varepsilon \alpha - \alpha^2=\alpha^2-2\varepsilon^2$$
as claimed. 
\end{proof}

\begin{thm}\label{2.4}
If $\operatorname{ch}(k)=2$, then $\varepsilon=0$, $\alpha \ne 0$, $\beta = \alpha^2$, and $\gamma = \alpha^7$. 
\end{thm}

\begin{proof}
It suffices to show that $\gamma =\alpha^7$. We choose elements $\varphi=\sum_{i=0}^\infty \varphi_it^i, \psi=\sum_{i=0}^\infty \psi_it^i$ of $A$ where $\varphi_i, \psi_i \in k$, so that 
\begin{center}
$\ \ \ \ (\#)\ \ \ \ \ g^2 = f^2\varphi + fg \psi.$
\end{center}
Compare the coefficients of $\{t^{i}\}_{51 \le i \le 60}$ in both sides of Equation $(\#)$. Then, since $g = t^{26} + \tau t^{33}$, we get the following table.
$$
\begin{tabular}{|c|l|l|}
\hline
$\text{Monomials}$   & $\text{Equation}~g^2=f^2 \varphi + fg \psi$ &  $\text{Results}$ \\ \hline
$t^{51}$ & $0=\varphi_{11}+\psi_5$ & {} \\ \hline 
$t^{52}$ & $1=\psi_5 \alpha$   & $\psi_5=\varphi_{11}=\frac{1}{\alpha}$ \\ \hline 
$t^{53}$ & $0=\varphi_{11}\alpha^2 +\alpha^2 \psi_5$ &{}  \\ \hline
$t^{54}$ & $0=0$ \\ \hline
$t^{55}$ & $0=\varphi_{15}+\varphi_{11}\alpha^4$ & $\varphi_{15}=\alpha^3$ \\ \hline
$t^{56}$ & $0=\varphi_{16} + \psi_{10}$   & $\psi_{10}=\varphi_{16}$ \\ \hline
$t^{57}$ & $0=\varphi_{15}\alpha^2 + \psi_{11} + \psi_{10}\alpha$ & {} \\ \hline
$t^{58}$ & $0=\varphi_{16}\alpha^2+\alpha \psi_{11}+\alpha^2\psi_{10} + \psi_5 (\tau + \gamma)=\alpha \psi_{11} + \frac{1}{\alpha}(\tau + \gamma)$ & $\psi_{11} = \frac{\tau+\gamma}{\alpha^2}$ \\ \hline
$t^{59}$ & $0=\alpha^4 \varphi_{15} + \alpha^2\psi_{11} + \tau \alpha \psi_5=\alpha^7+\alpha^2 {\cdot}\frac{\tau + \gamma}{\alpha^2} + \tau \alpha {\cdot}\frac{1}{\alpha}=\alpha^7+\gamma$ & $\gamma =\alpha^7$ \\ \hline$t^{60}$ & $0=\psi_5\beta\tau+\varphi_{20} + \varphi_{16}\beta^2$   & $\varphi_{20}+\varphi_{16}\beta^2= \frac{\beta \tau}{\alpha}$ \\ \hline
\end{tabular} \vspace{0.3em}
$$
The third column of the table consists of the results of computation of the coefficients $\varphi_i, \psi_j$ and eventually shows  $\gamma = \alpha^7$. 
\end{proof}

We finish the case where $\operatorname{ch}(k)=2$.

\begin{cor}\label{2.5}Let $\calY_A$ denote the set of Ulrich ideals of $A$ which possess the data $(a,b) = (20, 26)$.
Assume that $\operatorname{ch}(k)=2$. Then $$\calY_A=\{(t^{20} + \alpha t^{21} + \alpha^2 t^{22} + \alpha^7 t^{27} + \delta t^{33}, t^{26} + \tau t^{33}) \mid \alpha, \delta, \tau \in k, \alpha \ne 0 \}.$$
\end{cor}

\begin{proof}
Let $\alpha, \delta, \tau \in k, \alpha \ne 0$ and let $f=t^{20} + \alpha t^{21} + \alpha^2 t^{22} + \alpha^7 t^{27} + \delta t^{33}$, $g=t^{26} + \tau t^{33}$. We set $I = (f,g)$. Thanks to Theorem \ref{2.4}, it suffices to show $I \in \calX_A$. Since $A/I \cong I/(f)$ by Lemma \ref{2.1} (1), it is enough to show $g^2 \in (f^2,fg)$. Notice that the table $(\#)$ given in the proof of Theorem \ref{2.4} actually shows that if we set $$\varphi = \frac{1}{\alpha}t^{11} + \alpha^3 t^{15}\ \ \text{and}\ \ \psi = \frac{1}{\alpha}t^5+\frac{\tau + \alpha^7}{\alpha^2}t^{11},$$ we then have $$g^2 \equiv~f^2\varphi  + fg \psi~\mod~t^{60}V,$$ which implies $g^2 \in fI$, because $t^{60}V=f\fkc$ and $\fkc \subseteq I$ by Lemma  \ref{2.1} (1). 
\end{proof}

\begin{ex}[{cf. Theorem \ref{thm1}}]
Let $I=(t^{20}+t^{21}+t^{22}+t^{27}, t^{26})$. Then $I$ is an Ulrich ideal of $A$ if and only if $\operatorname{ch}(k)=2$. 
\end{ex}

\begin{proof}
It is enough to prove the {\em only if} part. Suppose that $I$ is an Ulrich ideal of $A$ and consider $\xi = \frac{g}{f}= \frac{t^{26}}{t^{20}+t^{21}+t^{22}+t^{27}}$. We then have
$$\xi = t^6-t^7+t^9-t^{10}+t^{12}-2t^{13} + \rho$$ with $\rmo(\rho) \ge 14$. Therefore
$$\frac{g^2}{f}=g\xi = t^{32} - t^{33} + t^{35}-t^{36}  + t^{38} - 2t^{39} + t^{26}\rho$$
which forces that $-2 =0$, because $g\xi \in I$ and $39 \not\in H$. Hence, $\operatorname{ch}(k)=2$.
\end{proof}

\section{The case where $H=\left<5, 11 \right>, (a,b)=(20,26)$, and $\operatorname{ch}(k)\ne2$}
Let $\alpha, \beta, \gamma, \delta, \varepsilon, \tau \in k$ and let
$$
f=t^{20} +\alpha t^{21}+\beta t^{22} +\gamma t^{27}+\delta t^{33},\ \ g= t^{26} +\varepsilon t^{27} + \tau t^{33}.
$$
We assume that $g^2=f^2\varphi + fg\psi$ for some $\varphi=\sum_{i=0}^\infty\varphi_it^i, \psi=\sum_{i=0}^\infty\psi_it^i \in A$ ($\varphi_i, \psi_i \in k$). Then, comparing the coefficients of $\{t^i\}_{0 \le i \le 60}$ in both sides of the equation
$$
g^2=f^2{\cdot}\sum_{i=0}^\infty\varphi_it^i+fg{\cdot}\sum_{i=0}^\infty\psi_it^i,
$$
we have $\psi_0=0, \varphi_0=\varphi_5=\varphi_{10}=0$ and the following table $(\#)$.
{\footnotesize
$$
\begin{tabular}{|c|l|}
\hline
$\text{Monomials}$   &  $\text{Results}$ \\ \hline
$t^{51}$ & $0=\psi_{5}+\varphi_{11}$ \\ \hline 
$t^{52}$ & $1=(\varepsilon + \alpha)\psi_5 + 2\alpha \varphi_{11}=(\varepsilon -\alpha)\psi_5$   \\ \hline 
$t^{53}$ & $2\varepsilon=\psi_5(\varepsilon \alpha + \beta) + \varphi_{11}(2\beta+\alpha^2)=\psi_5(\varepsilon \alpha - \beta - \alpha^2)$  \\ \hline
$t^{54}$ & $\varepsilon^2 =\psi_5{\cdot}\varepsilon \beta +2\varphi_{11}\alpha \beta = \psi_5\beta(\varepsilon - 2\alpha)$ \\ \hline
$t^{55}$ & $0=\varphi_{15}+\varphi_{11}\beta^2$ \\ \hline
$t^{56}$ & $0=\varphi_{16} + \varphi_{15}{\cdot}2\alpha +\psi_{10}$   \\ \hline
$t^{57}$ & $0=\psi_{11}+\psi_{10}(\varepsilon + \alpha) + \varphi_{16}{\cdot}2\alpha+\varphi_{15}(\alpha^2+2\beta)$  \\ \hline
$t^{58}$ & $0=\psi_{11}(\varepsilon + \alpha)+ \psi_{10}(\alpha \varepsilon + \beta)+ \psi_5(\tau + \gamma) + \varphi_{16}(\alpha^2+ 2 \beta)+2\alpha \beta \varphi_{15}+2\gamma \varphi_{11}$ \\ \hline
$t^{59}$ & $2\tau=\psi_{11}(\alpha \varepsilon + \beta)+\psi_{10}\beta \varepsilon +\psi_5(\varepsilon \gamma + \alpha \tau)+ \varphi_{16}{\cdot}2\alpha \beta +\varphi_{15}\beta^2 + \varphi_{11}{\cdot}2\alpha \gamma$  \\ \hline
$t^{60}$ & $2\varepsilon \tau=\beta \varepsilon \psi_{11} ++ \psi_5{\cdot}\beta \tau +  \varphi_{20} + \varphi_{16}\beta^2 + 2\beta \gamma{\cdot}\varphi_{11}$   \\ \hline
\end{tabular} \vspace{0.3em}
$$
}
Hence, we have the following.

\begin{lem}
$\varepsilon \ne \alpha, \ \psi_5=\frac{1}{\varepsilon -\alpha}, \ \psi_5(\varepsilon \alpha -\beta -\alpha^2)=2 \varepsilon$,\  and \ $\psi_5\beta(\varepsilon -2\alpha)=\varepsilon^2$.
\end{lem}

We now furthermore assume that $\operatorname{ch}(k) \ne 2$ and that $\varepsilon \ne \alpha$, $3 \varepsilon = 2\alpha$, and $\beta = \alpha^2-2\varepsilon^2$. We then have $\varepsilon \ne0, \alpha = \frac{3}{2}\varepsilon$, and $\beta = \frac{1}{4}\varepsilon^2$. Therefore, thanks to the table $(\#)$, we see
\begin{eqnarray*}
\varphi_{11} &=& -\psi_5~=~\frac{2}{\varepsilon}\\ 
\varphi_{15}&=&-\frac{1}{8}\varepsilon^3\\
\varphi_{16}&=&\frac{3}{8}\varepsilon^4-\psi_{10}\\
\psi_{11}&=&\frac{1}{2}\varepsilon\psi_{10} - \frac{25}{32}\varepsilon^5\\
\vspace{-2em}
\end{eqnarray*}
where the last equality follows from the comparison of the coefficients of $t^{57}$. Hence, considering the coefficients of $t^{58}$, we get
$$
\frac{2}{\varepsilon}(\tau -\gamma) =\frac{1}{4}{\varepsilon^2}\psi_{10}-\frac{65}{64}\varepsilon^6,
$$
while we have by the coefficients of $t^{59}$ that 
$$
5\tau-4\gamma =\psi_{10}{\cdot}\frac{3}{8}\varepsilon^3 - \frac{35}{32}\varepsilon^7.
$$
Therefore, we get the following.

\begin{prop}\label{lemma3}
Assume that $\operatorname{ch}(k) \ne 2$ and that $\varepsilon \ne \alpha$, $3 \varepsilon = 2\alpha$, $\beta = \alpha^2-2\varepsilon^2$. Then we have
$$
\gamma=2\tau - \frac{55}{128}\varepsilon^7, \ \ \psi_{10}= \frac{15}{2}\varepsilon^4 - \frac{8 \tau}{\varepsilon^3}.$$
Hence, $\psi_{11}= \frac{95}{32}\varepsilon^5 - \frac{4\tau}{\varepsilon^2}$, $\varphi_{16}=-\frac{57}{8}\varepsilon^4 + \frac{8\tau}{{\varepsilon}^3}$, and $\varphi_{20}=\varepsilon \tau +\frac{17}{128}\varepsilon^8$. 
\end{prop}

We are now ready to finish the case where $\operatorname{ch}(k) \ne 2$. Our conclusion is the following.

\begin{thm}\label{thm1} Let $\calY_A$ denote the set of Ulrich ideals of $A$ which possess the data $(a,b) = (20, 26)$.
Assume that $\operatorname{ch}(k)\ne 2$. Then
{\footnotesize
$$
\calY_A=\{(t^{20} + \alpha t^{21} + \beta t^{22} + \gamma t^{27} + \delta t^{33}, t^{26} + \varepsilon t^{27} + \tau t^{33}) \mid \alpha, \beta, \gamma, \delta, \varepsilon, \tau \in k, \varepsilon \ne \alpha, 3\varepsilon = 2 \alpha, \beta = \alpha^2 - 2\varepsilon^2, \gamma = 2\tau - \frac{55}{128}\varepsilon^7\}.
$$
 }
\end{thm}

\begin{proof}
Let $\alpha, \beta, \gamma, \delta, \varepsilon, \tau \in k$ and let
$$
f=t^{20} +\alpha t^{21}+\beta t^{22} +\gamma t^{27}+\delta t^{33},\ \ g= t^{26} +\varepsilon t^{27} + \tau t^{33}.
$$ If $(f,g)$ is an Ulrich ideal of $A$, then by Proposition \ref{lemma2} we have $\varepsilon \ne \alpha$, $3 \varepsilon = 2\alpha$ and $\beta = \alpha^2-2\varepsilon^2$. Therefore, it follows from Proposition \ref{lemma3} that $\gamma = 2\tau - \frac{55}{128}\varepsilon^7$, because $\operatorname{ch}(k) \ne 2$.

Conversely, let $I = (f,g)$ where $f=t^{20} +\alpha t^{21}+\beta t^{22} +\gamma t^{27}+\delta t^{33}$ and $g= t^{26} +\varepsilon t^{27} + \tau t^{33}$ with $\alpha, \beta, \gamma, \delta, \varepsilon, \tau \in k$ such that $\beta = \alpha^2-2\varepsilon^2$ and $\gamma = 2\tau - \frac{55}{128}\varepsilon^7$. Let
\begin{eqnarray*}
\varphi &=& \frac{2}{\varepsilon}t^{11} - \frac{1}{8}\varepsilon^3t^{15} + (\frac{8\tau}{\varepsilon^3}-\frac{57}{8}\varepsilon^4)t^{16} + (\varepsilon \tau + \frac{17}{128}\varepsilon^3)t^{20},\\ 
\psi&=&-\frac{2}{\varepsilon}t^5 + (\frac{15}{2}\varepsilon^4 - \frac{8\tau}{\varepsilon^3})t^{10} + (\frac{95}{32}\varepsilon^5- \frac{4\tau}{\varepsilon^2})t^{11}.\\
\vspace{-2em}
\end{eqnarray*}
Then the table $(\#)$ shows that $$g^2 \equiv f^2 \varphi + fg \psi~\mod~t^{61}V.$$
On the other hand, by Lemma  \ref{2.1} (1) we have $\fkc = t^{40}V \subseteq I$. Hence, $t^{60}V=f\fkc \subseteq fI$, so that $g^2 \in (f^2,fg)$. Thus, $I$ is an Ulrich ideal of $A$ by Lemma  \ref{2.1}. This completes the proof of the case where $\operatorname{ch}(k) \ne 2$.
\end{proof}

Let us give a brief note about the proof of Theorem \ref{main}.

\begin{proof}[Proof of Theorem {\rm \ref{main}}]
Thanks to Corollary \ref{2.5}, we may assume that $\operatorname{ch}(k) \ne 2$. Let $\alpha, \beta, \gamma, \delta, \varepsilon, \tau \in k$ such that $$(\rmC_1)\ \ \ \  \varepsilon \ne \alpha, \ 3\varepsilon = 2 \alpha, \ \beta = \alpha^2 - 2\varepsilon^2, \ \gamma = 2\tau - \frac{55}{128}\varepsilon^7.$$ Then, setting $\varepsilon'= \frac{1}{2}\varepsilon$, we have $\varepsilon' \ne 0$ and 
$$(\rmC_2)\ \ \ \ \ \varepsilon=2\varepsilon',\  \alpha = 3\varepsilon', \ \beta ={ \varepsilon'}^2, \ \gamma =2\tau - 55{\varepsilon'}^7.$$ It is certain that once we choose $\delta, \varepsilon', \tau \in k$ so that $\varepsilon' \ne 0$ and set $\varepsilon, \alpha, \beta, \gamma$ like Condition $(\rmC_2)$, the elements $\alpha, \beta, \gamma, \delta, \varepsilon, \tau \in k$ satisfy Condition $(\rmC_1)$ required in Theorem \ref{thm1}, and therefore $(t^{20} + 3\varepsilon'  t^{21} + {\varepsilon'}^2 t^{22} + (2\tau - 55{\varepsilon'}^7)t^{27} + \delta t^{33}, t^{26} + 2{\varepsilon'} t^{27} + \tau t^{33}) \in \calX_A$.

Let us check the second assertion. Let 
\begin{eqnarray*}
f&=&t^{20} + 3\varepsilon  t^{21} + \varepsilon^2 t^{22} + (2\tau - 55 \varepsilon^7)t^{27} + \delta t^{33},\\ 
g&=& t^{26} + 2\varepsilon t^{27} + \tau t^{33},\\
f_1&=&t^{20} + 3\varepsilon_1 t^{21} + {\varepsilon_1}^2 t^{22} + (2\tau_1 - 55 {\varepsilon_1}^7)t^{27} + \delta_1 t^{33},\\
g_1&=& t^{26} + 2\varepsilon_1 t^{27} + \tau_1 t^{33}.\\
\end{eqnarray*}
where $\delta, \varepsilon, \tau, \delta_1, \varepsilon_1, \tau_1 \in k$ such that $\varepsilon \ne 0, \varepsilon_1 \ne 0$, and assume that $I =(f,g)=(f_1,g_1)$.

\bigskip
\noindent
We then have 
{\footnotesize
$$3(\varepsilon -\varepsilon_1)t^{21} + (\varepsilon^2-{\varepsilon_1}^2)t^{22} + \left\{2(\tau -\tau_1)-55(\varepsilon^7 - {\varepsilon_1}^7)\right\}t^{27} + (\delta -\delta_1)t^{33}, \  \ 2(\varepsilon-\varepsilon_1)t^{27}+(\tau - \tau_1)t^{33} \in I.$$ 
}
Therefore, since the images of $t^{21}, t^{22},t^{27}, t^{33}$ form a part of a $k$-basis of $A/I$ as is shown in the proof of Lemma \ref{2.1},  we readily get $\delta = \delta_1$, $\tau = \tau_1$. We have $\varepsilon = \varepsilon_1$, because $3(\varepsilon -\varepsilon_1) = 2(\varepsilon -\varepsilon_1)=0$. This finishes the proof of Theorem \ref{main}.
\end{proof}

\begin{cor}[{\cite{GOTWY}}]
The ring $A$ contains no Ulrich ideals generated by monomials in $t$.
\end{cor}

\begin{cor}
Let 
$f=t^{20} +3t^{21}+t^{22} +\gamma t^{27}+\delta t^{33}$ and $g= t^{26} + 2 t^{27} + \tau t^{33}$ with $\gamma, \delta, \tau \in k$ such that $\gamma = 2\tau - 55$. Then, $I=(f,g)$ is an Ulrich ideal of $A$ in any characteristic.
\end{cor}

\begin{cor}
Suppose that $\operatorname{ch}(k) = 3$ and let 
$f=t^{20} +\varepsilon^2 t^{22} +\gamma t^{27}+\delta t^{33}$ and $g= t^{26} +\varepsilon t^{27} + \tau t^{33}$ with $\varepsilon, \gamma, \delta \in k$ such that $\varepsilon \ne 0$ and $\gamma = 2\tau + \varepsilon^7$. Then $I=(f,g)$ is an Ulrich ideal of $A$. In particular, $(t^{20} + t^{22}, t^{26}+t^{27}+t^{33}) \in\calX_A$.
\end{cor}


\end{document}